\documentclass[12pt,leqno]{amsart} 
\usepackage[numbers]{natbib} 
\usepackage{amsmath,amsthm}
\usepackage{amssymb}
\usepackage{accents}
\usepackage[utf8]{inputenc}

\renewcommand{\widehat}{\hat}
\usepackage{url}
\usepackage{todonotes}
\usepackage[colorlinks={true},
backref,pdfstartview={FitBH},pdfview={FitBH},bookmarks={true}]{hyperref}
\global\hbadness=500
\global\tolerance=200
\newtheorem{theorem}{Theorem}[section]
\newtheorem{lemma}[theorem]{Lemma}

\newtheorem{claim}[theorem]{Claim}
\newtheorem{question}[theorem]{Question}

\theoremstyle{definition}
\newtheorem{definition}[theorem]{Definition}

\theoremstyle{remark}

\newtheorem{remark}[theorem]{Remark}
\numberwithin{equation}{theorem}

\newcommand{\converge}{\!\!\downarrow}
\newcommand{\diverge}{\!\!\uparrow}

\renewcommand{\phi}{\varphi}

\newlength{\dhatheight}

\DeclareMathOperator \dom{dom}

\title[Rado Path Decomposition Theorem]{The Rado Path Decomposition
  Theorem}
  
\author[Cholak]{Peter~A.~Cholak}

\address[Cholak]{Department of Mathematics\\ University of Notre Dame\\ 
  Notre Dame, IN 46556-5683}

\email[Cholak]{Peter.Cholak.1@nd.edu}

\urladdr[Cholak]{\url{http://www.nd.edu/~cholak}}

\author[Igusa]{Gregory Igusa}

\email[Igusa]{gregigusa@gmail.com}

\author[Patey]{Ludovic Patey}

\address[Patey]{Institut Camille Jordan\\
Universit\'e Claude Bernard Lyon 1\\
43 boulevard du 11 novembre 1918\\
F-69622 Villeurbanne Cedex}

\email[Patey]{Ludovic.Patey@computability.fr}

\urladdr[Patey]{\url{http://ludovicpatey.com}}

\author[Soskova]{Mariya I. Soskova}

\address[Soskova]{Department of Mathematics\\ University of Wisconsin--Madison\\
480 Lincoln Dr, Madison WI 53706}

\email[Soskova]{msoskova@math.wisc.edu}

\urladdr[Soskova]{\url{http://www.math.wisc.edu/~msoskova}}

\author[Turetsky]{Dan Turetsky} \address[Turetsky]{Department of
  Mathematics\\ Victoria University of Wellington \\New Zealand}
\email[Turetsky]{dan.turetsky@vuw.ac.nz}
\urladdr[Turetsky]{\url{http://tinyurl.com/dturetsky}}

\date{\today}

\thanks{This work was partially supported by a grant from the Simons
  Foundation (\#315283 to Peter Cholak).  Igusa was partially
  supported by EMSW21-RTG-0838506. Soskova's research was supported by
  National Science Fund of Bulgaria grant \#01/18 from 23.07.2017 and
  by National Science Foundation grant DMS-1762648. }


\begin{document}



\maketitle

\section{Introduction}

Fix $c:[\mathbb{N}]^2 \rightarrow r$, an $r$-coloring of the pairs of
natural numbers.  An ordered list of distinct integers,
$a_0, a_1, a_2 \ldots a_{i-1}, a_i, a_{i+1} \ldots$ is a
\emph{monochromatic path for color $k$}, if, for all $i\geq 1$,
$c(\{a_{i-1},a_i\})=k$.  The empty list is considered a path of any
color $k$.  Similarly, the list of one element, $a_0$, is also
considered a path of any color $k$. For any monochromatic path of
length two or more the color is uniquely determined.  Paths can be
finite or infinite.  Since all paths considered in this article are
monochromatic we will drop the word monochromatic.

\begin{definition}
  Let $c$ be an $r$-coloring of $[\mathbb{N}]^2$ ($[n]^2$). A
  \emph{path decomposition for $c$} is a collection of $r$ paths
  $P_0, P_1,\dots, P_{r-1}$ such that $P_j$ is a path of color $j$ and
  every integer (less than $n$) appears on exactly one path.
\end{definition}

Improving on an unpublished result of Erd\H{o}s, Rado \cite{MR0485504}
published a theorem which implies:

\begin{theorem}[Rado Path Decomposition, $\mathsf{RPD}$, or $\mathsf{RPD}$$_r$]\label{original}
  Every $r$-coloring of the pairs of natural numbers has a path
  decomposition.
\end{theorem}

In Section~\ref{proofs}, we provide three different proofs of this
result.  The first proof makes use of an ultrafilter on the natural
numbers.  This ultrafilter proof is clearly known but has only
recently appeared in print, see Lemma~2.2 of \cite{Elekes:2015wd}.  The
remaining proofs are interesting new modifications of the ultrafilter
proof.

All of the proofs presented are highly non-computable.  In
Section~\ref{sec:halting}, we show that a non-computable proof is
necessary. A coloring $c:[\mathbb{N}]^2 \rightarrow r$ is
\emph{stable} if and only if $\lim_y c(\{x, y\})$ exists
for every $x$.  We show there is a computable stable $2$-coloring $c$
of $[\mathbb{N}]^2$ such that any path decomposition for $c$ computes
the halting set. In Section~\ref{sec:2}, we give a non-uniform proof
of the fact that the halting set can compute a path decomposition for
any computable $2$-coloring.

In Section~\ref{uniform} we show that if our primary $\Delta^0_2$
construction from Section~\ref{sec:2} fails then it is possible to
find a path decomposition which is as simple as possible: one path is
finite and the other computable.  But even with this extra knowledge,
we show, in Theorem~\ref{nonuniform}, that there is no uniform proof
of the fact that the halting set can compute a path decomposition for
any computable $2$-coloring.  In Theorem~\ref{nonuniform2}, we improve
this to show no finite set of $\Delta^0_2$ indices works.


In Section~\ref{sec:logic} we show that the halting set can also
compute a path decomposition for stable colorings with any number of
colors.  The rest of Section~\ref{sec:logic} discusses Rado Path
Decomposition within the context of mathematical logic and, in
particular, from the viewpoint of computability theory and reverse
mathematics.  In Section~\ref{sec:ramsey}, we discuss two differences
between Rado Path Decomposition and Ramsey Theorem for pairs.

Most of the sections can be read in any order, although
Section~\ref{uniform} relies on Section~\ref{sec:2}, and
Section~\ref{sec:logic} relies on Section~\ref{proofs}.

Our notation is standard. Outside of Sections~\ref{mov} and
\ref{sec:logic}, and possibly Section~\ref{uniform}, our use of
computability theory and mathematic logic is minimal and very
compartmentalized.  One needs to be aware of the halting set and the
first few levels of the arithmetic hierarchy. A great reference for
this material is \citet{MR2920681}.  For more background in reverse
mathematics, including all notions discussed in Sections~\ref{mov} and
~\ref{sec:logic}, we suggest \citet{MR3244278}.

Our interest in the $\mathsf{RPD}$ was sparked by \citet{Soukup:2015wv}. Thanks!

\subsection{$\mathsf{RPD}$ within the framework of computable
  combinatorics}\label{mov}

In computable combinatorics we consider combinatorics principles as
instances-solutions pairs and compare the computational power of
solutions.  With $\mathsf{RPD}_r$, an instance is $r$-coloring and the
solution is a path decomposition.  With Ramsey's theorem for pairs and
$r$ colors, $\mathsf{RT}^2_r$, the instance is an $r$-coloring and the
solution is a homogenous set, see Remark~\ref{RT22}.  Another classic
combinatorics problem is Weak K\"{o}nig's Lemma, $\mathsf{WKL}$; where
an instance is an infinite subtree of $2^{< \omega}$ and a solution is
an infinite path through the tree.

There are many ways one can compare the computational power of
solutions.  For example, since the halting set computes an infinite
path through every computable instance of $\mathsf{WKL}$,
Theorem~\ref{halting}, implies that every solution to $\mathsf{RPD}_r$
computes a solution to every computable instance of $\mathsf{WKL}$.
So $\mathsf{RPD}_r$ is stronger than $\mathsf{WKL}$.  While we show
$\mathsf{RPD}_r$ is strictly stronger than $\mathsf{RT}^2_r$, the
relationship between their solutions is not as straightforward.  One
can consider a Turing ideal, $\mathcal{I}$, an \emph{ideal model} of a
combinatorics principle if every instance in $\mathcal{I}$ has a
solution in $\mathcal{I}$.  Our work shows that every ideal model of
$\mathsf{RPD}_r$ is a model of $\mathsf{RT}^2_r$ but the converse
fails.  We also show that solutions to computable instances of
$\mathsf{RPD}_2$ cannot compute solutions to computable instances of
$\mathsf{RT}^2_r$.  A postive answer to Question~\ref{PA} would imply
that computable instances of $\mathsf{RPD}_3$ can compute solutions to
computable instances of $\mathsf{RT}^2_r$.

Another way to measure the strength of these principles is as a
statement in second order arithmetic. Here we think of combinatorics
principles as a set existence theorem, the combinatorics principle
implies that a solution exists for every instance that exists. Here we
show that over $\mathsf{RCA}_0$, the system corresponding to the
existence of the computable sets, that $\mathsf{RPD}_r$ is equivalent
to $\mathsf{ACA}_0$, the system corresponding to the existence of the
arithmetic sets.

There are many more combinatorics principles and ways one can compare
the computational power of combinatorics principles.  We cannot
discuss them all here, again we suggest \citet{MR3244278} as a
starting point.

\section{Some Proofs of $\mathsf{RPD}$}\label{proofs}

In this section we will provide several proofs of $\mathsf{RPD}$. We need to
start with some notation and definitions.  The union of pairwise
disjoint sets is written as $X_0 \sqcup X_1 \ldots \sqcup X_i$.  Two
sets are
equal modulo finite, $X=^*Y$, if and only if their symmetric
difference $ X\triangle Y$ is finite.  If
$X_0 \sqcup X_1 \ldots \sqcup X_i =^* Z $ then the $X_j$'s are
pairwise disjoint and their union is equivalent modulo finite to $Z$.
If $Z= \mathbb{N}$ then $X_0 \sqcup X_1 \ldots \sqcup X_i$ is almost
form  a   
partition of $\mathbb{N}$, that is there is a finite set $F$ such that
$F\sqcup X_0 \sqcup X_1 \ldots \sqcup X_i = \mathbb{N}$.

\begin{definition}
  A collection $\mathcal{U}$ of subsets of $\mathbb{N}$ is an
  \emph{ultrafilter} (on $\mathbb{N}$) if and only if
  $\emptyset\notin \mathcal{U}$, $\mathcal{U}$ is closed under
  superset, $\mathcal{U}$ is closed under finite intersections, and, for all
  $X \subseteq \mathbb{N}$, either $X \in \mathcal{U}$ or its
  complement $\overline{X} \in \mathcal{U}$. An ultrafilter is
  \emph{non-principal} if and only if, for all $a \in \mathbb{N}$,
  $\{a\} \notin \mathcal{U}$.
\end{definition}

We will call a subset $X$ of $\mathbb{N}$ \emph{large} if and only if
$X \in \mathcal{U}$.

\begin{remark} \label{remark1}
  The two key facts that we will need about a non-principal
  ultrafilter $\mathcal{U}$ are as follows.
  \begin{enumerate}
  \item $\mathcal{U}$ does not have finite members. (This statement
    follows by an easy induction on the size of the finite set.)
  \item If
  $$X_0 \sqcup X_1 \ldots \sqcup X_i =^* \mathbb{N}$$
  then \emph{exactly one} of the $X_j$ is large. (No more than one of
  these sets of can be large, because if $X_{j_0}$ and $X_{j_1}$ are
  distinct then they have an empty intersection. Assume that for all
  $j$, $\overline{X}_j \in \mathcal{U}$.  It follows that
  $\bigcap_{j\leq i} \overline{X}_j =^* \emptyset \in \mathcal{U}$.
  But no finite set can be a member of a non-principal ultrafilter,
  giving us the desired contradiction.)  \end{enumerate}
\end{remark}

For the rest of this section a coloring
$c:[\mathbb{N}]^2 \rightarrow r$ will be fixed.

\subsection{Ultrafilter Proof}

The existence of a non-principal ultrafilter on the natural numbers is
a strong assumption that unfortunately cannot be shown in Zermelo
Fraenkel set theory, see \citet{MR0176925}; the axiom of choice is
sufficient see \citet{MR1940513}.
Nevertheless, we give a proof of $\mathsf{RPD}$ that uses this assumption,
because we believe that it provides insight into the combinatorics of
this statement. Later in this section we will give alternative proofs
of $\mathsf{RPD}$ that do not use a non-principal ultrafilter.

Let $\mathcal{U}$ be a non-principal ultrafilter.  We will denote the
set of neighbors of $m$ with color $i$ by
  $$N(m,i) = \{n : c(\{m,n\}) = i \}.$$ 
  
  Note that $N(m,i)$ is computable in our coloring $c$. Furthermore,
  if we fix $m$ then the sets $N(m,i)$ where $i<r$ form almost form a
  partition of $\mathbb{N}$, just $m$ is missing.  By
  Remark~\ref{remark1} for every $m$ there is a unique $j < r$ such
  that $N(m,j)$ is {large}. Let
  $A_j = \{m : N(m,j) \text{ is large}\}$. The sets $A_j$ where $j<r$
  also partition $\mathbb{N}$.  If $m \in A_j$ then we will say that
  $m$ has color $j$. It follows that every natural number is assigned
  in this way a unique color.
  
  For any pair of points $m<n$ in $A_j$, $N(m,j) \cap N(n,j) $ is
  large. So there are infinitely many $v \in N(m,j) \cap N(n,j) $.
  For all such $v$, $c(m,v)=c(v,n)=j$. Note that any such $v$ is
  likely much larger than $m$ and $n$ and not necessarily in
  $A_j$. 

  \begin{proof}[Construction]\label{construction}
    We will construct our path decomposition
    $P_0, P_1, \dots, P_{r-1} $ in stages.  Let $P_{j,0}=\emptyset$
    for all $j<r$. The path $P_{j,0}$ is the empty path of color $j$.
    Assume that for each $j < r$, $P_{j,s}$ is a finite path of color
    $j$ such that if $P_{j,s}$ is nonempty then its last member is of
    color $j$ (i.e.\ in $A_j$).  Assume also that every $t< s$
    appears in one of the $P_{j,s}$.  If $s$ already appears in one of
    the $r$ paths, then let $P_{j,s+1}=P_{j,s}$ for all $j<r$.
    Otherwise, $s$ has some color $k$.  For $j\neq k$, let
    $P_{j,s+1}=P_{j,s}$. If $P_{k,s}$ is empty, then let
    $P_{k,s+1} = \{s\}$.  Otherwise, let $e$ be the end of the path
    $P_{k,s}$.  There is a $v$ not appearing in any of the finite
    paths $P_{j,s}$ such that $v \in N(e,k) \cap N(s,k)$.  Add $v$ and
    $s$ to the end of $P_{k,s}$ in that order to get $P_{k,s+1}$.  To
    complete the construction we set $P_j = \lim_s P_{j,s}$ for every
    $j<r$.  The desired path decomposition is given by
    $P_0, P_1, \dots. P_{r-1}$.
  \end{proof}

  The proof described above is very close to the well known
  ultrafilter proof of Ramsey's theorem for pairs. To illustrate this
  we include this proof below.  An infinite set $H$ is
  \emph{homogenous} for $c$ if and only if $c([H]^2) $ is
  constant. Ramsey's theorem for pairs is the statement that every
  $r$-coloring of the pairs of natural numbers has an infinite
  homogeneous set.

  \begin{remark}\label{RT22}[Proof of the existence of a homogenous set for $c$]
    Recall that $A_0, A_1, \dots A_{r-1}$ gives a partition of
    $\mathbb{N}$. Fix the unique $j$ such that $A_j$ is large. We can
    \emph{thin} $A_j$ to get an infinite homogenous set $H$ of color
    $j$ as follows: we build an infinite sequence
    $\{h_n\}_{n\in\mathbb{N}}$ of elements in $A_j$ by induction so
    that $H = \{ h_0, h_1, \ldots \}$ is as desired. Let $h_0$ be the
    least element of $A_j$. Suppose that we have constructed a
    homogeneous set $\{h_0,\dots, h_i\}\subseteq A_j$.  Since
    $A_j\cap \bigcap_{k\leq i}N(h_k,j)$ is the finite intersection of
    large sets, it is also large and hence infinite. We define
    $h_{i+1}$ to be the least member of
    $A_j \cap \bigcap_{k\leq i}N(h_k,j)$ that is larger than $h_i$.
  \end{remark}

  \subsection{Cohesive Proof}\label{Cohesive_Proof_Subsection}

  As noted above, we would like to remove the use of the non-principal
  ultrafilter from the proof of $\mathsf{RPD}$. For this we will extract the
  specific relationship that $\mathcal{U}$ had with the sets $N(m,j)$.

 \begin{remark}\label{largeness}
   Reflecting on the above construction, we see that the important
   things about largeness were that
   \begin{enumerate}
   \item for every $m$ there is a unique $j < r$ such that $N(m,j)$ is
     large,
   \item large sets are not finite, and
   \item the intersection of two large sets is large.
   \end{enumerate}
  
 \end{remark}

  \begin{definition}
    An infinite set $C$ is \emph{cohesive} with respect to the
    sequence of sets $\{X_n\}_{n\in \mathbb{N}}$ if and only if for
    every $n$ either $C \subseteq^* X_n$ or
    $C\subseteq^* \overline{X_n}$.
  \end{definition}

  \begin{lemma}
    There is a set $C$ that is cohesive with respect to the sequence
    $\{N(m,j)\}_{j<r, m\in\mathbb{N}}$.
  \end{lemma}
   
  \begin{proof}
    Once again we will use a stagewise construction.  We will
    construct two sequences of sets: $\{C_s\}_{s\in \mathbb{N}}$ and
    $\{R_s\}_{s\in \mathbb{N}}$. The first sequence will be increasing
    and the second decreasing with respect to the subset relation.
    Start with $C_0 = \emptyset$ and $R_0 = \mathbb{N}$.  Fix some
    indexing of all pairs $\langle m,i \rangle$. Inductively assume
    that, for all $\langle m,i \rangle < s$, 
    either $R_s \subseteq N(m,i)$ or
    $R_s \subseteq \overline{N(m,i)}$, $C_s$ is finite, $R_s$ is
    infinite, and $C_s$ and $R_s$ are disjoint. At stage $s+1$, let
    $c_{s}$ be the least element of $R_s$.  Let
    $C_{s+1} = C_s \cup \{c_s\}$.  Assume that
    $s = \langle m,i \rangle $.  Since $R_s$ is an infinite set, at
    least one of $R_s \cap N(m,i)$ or $R_s \cap \overline{N(m,i)}$ is
    infinite. 
    If $R_s \cap N(m,i)$ is infinite let
    $R_{s+1} = \big( R_s \cap N(m,i) \big) - \{c_s\}$.  Otherwise let
    $R_{s+1} = \big( R_s \cap \overline{N(m,i)} \big) - \{c_s\}$.
    $C = \lim_s C_s = \{ c_0, c_1, \ldots \}$ is the desired cohesive
    set.
  \end{proof}

  Fix such a set $C$. We can now redefine largeness by using $C$
  instead of an ultrafilter. Call a set $X$ \emph{large} if and only
  if $C \subseteq^* X$.  This new notion of largeness has the three
  key properties outlined above with respect to the sets $N(m,i)$: for
  every $m$ there is a unique $j < r$ such that $N(m,j)$ is large,
  because $C$ cannot be a subset of two disjoint sets, even if we
  allow a finite error; large sets are not finite, because $C$ is
  infinite; and the intersection of two large sets is large, because
  if $C \subseteq^* X$ and $C \subseteq^* Y$ then
  $C \subseteq^* X\cap Y$.  We can now repeat the original
  construction
  using this notion of largeness to produce a path decomposition.

  \subsection{Stable Colorings}
  Recall that a coloring $c$ is \emph{stable} if and only if for every
  $m$ the limit $\lim_n c(\{m,n\})$ exists.  Rephrasing this property
  in terms of sets of neighbors, we get that there is a unique
  $j < r $ such that $N(m,j)$ is 
  cofinite.  So to construct a path decomposition for stable colorings
  we do not even need a cohesive set. We can redefine large to mean
  cofinite and use once again the original construction.

  \subsection{Generic Path Decompositions}

  In this section we will provide a forcing-style construction of a
  path decomposition. To avoid confusion with our ultrafilter proof,
  our construction will use sequences of conditions rather than poset
  filters.

  \emph{Conditions} are tuples $(P_0, P_1 \ldots P_{r-1}, X)$ such
  that
  \begin{enumerate}
  \item $X\subseteq \mathbb{N}$ is infinite,
  \item $P_j$ is a finite path of color $j$ for every $j<r$,
  \item 
    no integer appears on more than one of the paths, and
  \item if $P_j$ is nonempty and $e_j$ is its last element then
    $X\subseteq^* N(e_j,j)$ (so $e_j$ has color $j$ with respect to
    $X$).
  \end{enumerate}
  It follows that
  $(\emptyset, \emptyset, \ldots, \emptyset, \mathbb{N})$ is a
  condition, because it trivially satisfies the third requirement.  A
  condition
  $(\widehat{P}_0, \widehat{P}_1, \ldots, \widehat{P}_{r-1},
  \widehat{X})$ \emph{extends} $(P_0, P_1, \ldots, P_{r-1}, X)$ if and
  only if, for all $j$, $P_j$ is an initial subpath of
  $\widehat{P}_j$, and $\widehat{X} \subseteq X$. Unlike Mathais
  forcing, the new elements of our paths $\widehat{P}_j$ need not be
  elements of $X$.

  Given a sequence of conditions
  $\langle{C}_i\rangle_{i\in \mathbb{N}}$ such that for every $i$,
  $C_{i+1}$ extends $C_i$, we think of this sequence as approximating
  a tuple of paths as follows.

  If $C_i = (P^i_0, P^i_1, \ldots, P^i_{r-1}, X^i)$, then the sequence
  $\langle{C}_i\rangle$ approximates the tuple of paths
  $(\tilde{P}_0, \tilde{P}_1, \ldots \tilde{P}_{r-1})$ where
  $\tilde{P}_j = \lim_i P^i_j$.

  Such a tuple of paths need not be a path decomposition, since it
  might happen that some integer does not appear on any of the limit
  paths. The purpose of the $X$ values in the conditions will be to
  ensure that the approximated paths do form a path decomposition if
  the sequence $\langle{C}_i\rangle$ is \emph{generic} (defined
  below).

  A set of conditions $\mathcal{D}$ is \emph{dense} if every condition
  is extended by a condition in $\mathcal{D}$.  A sequence
  $\langle{C}_i\rangle$ \emph{meets} $\mathcal{D}$ if there is some
  $i$ such that $C_i \in \mathcal{D}$.
  
  Given any collection of dense sets, a sequence is
  $\langle{C}_i\rangle$ \emph{generic} for that collection if it meets
  every $\mathcal{D}$ in that collection.  Note that if we have a
  countable collection of dense sets $\mathcal{D}_i$ then it is
  straightforward to build a generic sequence for that collection, by
  inductively choosing each $C_{i+1}$ to extend $C_i$ and be in
  $\mathcal{D}_{i+1}$.

  Let $\mathcal{D}_i$ be the set of conditions
  $(P_0, P_1, \ldots, P_{r-1}, X)$ such that $i$ is on some path
  $P_j$.  The lemma below shows that $\mathcal{D}_i$ is dense. Any
  generic for $\{\mathcal{D}_i\}$ gives a path decomposition for $c$.

   \begin{lemma}
     For every $i$ the set $\mathcal{D}_i$ is dense.
   \end{lemma}

   \begin{proof}
     Fix $i$ and a condition $(P_0, P_1, \ldots, P_{r-1}, X)$.  If $i$
     is on one of the paths $P_j$ then we are done.  Otherwise, $X$ is
     an infinite set, so there must be a $j$ such that $N(i,j) \cap X$
     is infinite. If $k\neq j$ then let $\tilde{P}_k = P_k$.  Let
     $\tilde{X} = X \cap N(i,j)$.  If $P_j$ is empty let $\tilde{P}_j$
     be $i$.  Otherwise let $e$ be the end of $P_k$. Since
     $(P_0, P_1, \ldots, P_{r-1}, X)$ is a condition, there is a $v$
     such that $v \in N(e,j)\cap N(i,j)$. Let $\tilde{P}_j$ be $P_j$
     with $v$ and $i$ added to the end in that order. It follows that
     $(\tilde{P}_0, \tilde{P}_1 , \ldots \tilde{P}_{r-1}, \tilde{X} )$
     is a condition in $\mathcal{D}_i$ extending
     $(P_0, P_1, \ldots, P_{r-1}, X)$.
   \end{proof}

   The generic construction is very much in the style of Rado's
   original proof.  

   \section{Path Decompositions which compute the halting
     set}\label{sec:halting}

   Recall the halting set
   $K = \{ e | (\exists s)\varphi_{e,s}(e) \converge\}$ is the set of
   codes $e$ for programs which, when started with input $e$, halt
   after finitely many steps. The halting set was one of the first
   examples of a set that is not computable.  The goal of this section
   is to show the following theorem.

  \begin{theorem}\label{halting}
    There is a computable stable $2$-coloring $c$ of $[\mathbb{N}]^2$
    such that any path decomposition of $c$ computes the halting set.
  \end{theorem}
 
  We devote the rest of the section to the proof of this theorem.  For
  colors we will use RED and BLUE.  Once again a coloring $c$ is
  \emph{stable} if, for all $m$, $\lim_n c(m,n)$ exists.

  We will give a computable stagewise construction for $c$.  The goal
  will be to construct $c$ so that:
  \begin{enumerate}
  \item The BLUE path in any path decomposition is infinite;
  
  \item Any path decomposition can compute the elements of $K$ via the
    following algorithm: If $e$ is a natural number then the
    construction will associate a marker $m_e$ to $e$ in a way that is
    computable from any path decomposition for $c$.
    We enumerate the BLUE and RED paths until all numbers $x\leq m_e$
    have appeared on one of the two paths. Let $t$ be the next element
    on the BLUE path. Then $e\in K$ if and only if $\varphi_{e,t}(e)$
    is defined (i.e.\ $\varphi_{e}(e)$ halts after $\leq t$ many
    steps).
  
  \end{enumerate}
  
  Each $x \in \mathbb{N}$ will have a default color.  Initially it
  will be BLUE.  The default color of a number might be changed once
  during the construction to RED.  At stage $s$, we will define
  $c(\{x,s\})$ for every $x<s$ and we will always set this value to be
  the current default color for $x$.  So our construction will produce
  a stable coloring. To achieve our first goal, it will be sufficient
  to ensure that for all elements of infinitely many intervals
  $[k,2k+1]$ the default color BLUE is never changed.  This is because
  if all elements in the interval $[k,2k+1]$ are colored BLUE with
  every greater number, then, in any path decomposition, the BLUE path
  must contain a node in this interval: if $m$ is in this interval and
  on a RED path then the next and previous nodes on this RED path must
  be a number less than $k$, so the RED path can only contain at most
  $k$ of the nodes in this interval. The length of this interval is
  $k+2$, so at least one of the nodes in this interval must be on the
  BLUE path.  This idea is reflected in the way we associate markers
  $m_e$ to elements $e$.
  
  We will say that a number $k$ is \emph{fresh} at stage $s$ if and
  only if $k$ is larger than any number mentioned/used at any stage
  $t$ where $t \leq s$.  All markers $m_e$ are initially undefined,
  i.e.\ $m_{e,0}\diverge$.  At each stage $s$ before we proceed with
  the definition of $c(x,s)$ for $x<s$ we first update the markers:
  for the least $e$ where $m_{e,s-1}$ is not defined, we will select a
  fresh number $k$ and define $m_{e,s} = 2k+2$.  (Note that this means
  that if $n$ is fresh after stage $s$ then $n > 2k+2$.)  Unless we
  say otherwise (see below) at all later stages $t$ we will keep
  $m_{e,t} = m_{e,s}$.  It will follow that $\lim_s m_{e,s} = m_e$
  exists.

  We also update the default colors as follows. For every $e<s$ we
  check if $\varphi_{e,s-1}(e) \diverge$,
  $\varphi_{e,s}(e) \converge$, and $m_{e,s}$ is defined. If so we
  change the default color of all $x \in [m_{e,s}, s+1]$ to RED and
  make all $m_{i,s}$ undefined for all $i> e$.  If we can show that
  this construction satisfies our first goal, then we can easily argue
  that it also satisfies the second: Fix any path decomposition and
  assume that $t$ is the first element on the BLUE path after all
  numbers $x\leq m_e$ have shown up on one of the two paths. Suppose
  further that $\varphi_e(e)$ halts in $s$ many steps. We must show
  that $t>s$. If at stage $s$ we have that $m_{e,s}$ is not defined
  then $t>m_e> s$. If $m_{e,s}$ is defined and we assume that $t<s+1$
  then the BLUE path cannot be extended below $m_{e,s}$ because
  everything below $m_{e,s}$ has already been covered by one of the
  two paths, and it cannot be extended above $s+1$ because everything
  in the interval $[m_{e,s},s+1]$ is RED with everything larger than
  $s+1$.
  It follows that the BLUE path is finite, contradicting our
  assumption.
  
  For every $e$ the value of the marker $m_{e,s}$ can be cancelled at
  most $e$ many times and then stays constant, so
  $\lim_s m_{e,s} = m_e$ does exist. It is furthermore computable from
  any path decomposition by the following procedure. The marker for
  $0$ is never cancelled, so $m_{0} = m_{0,1}$. If we know the value
  of $m_e$ then we run the construction until we see the first stage
  $t_0$ such that $m_{e} = m_{e, t_0}$. It follows that after stage
  $t_0$ we can cancel $m_{e+1}$ only for the sake of $e$.  We can also
  figure out if $e\in K$ by looking for the first $t_1$ on the BLUE
  path after all numbers $x\leq m_e$ have shown up on one of the two
  paths and checking whether or not $\varphi_e(e)$ halts in $t_1$
  steps. Let $t = \max(t_0, t_1)+1$.  We claim that
  $m_{e+1, t} = m_{e+1}$. If $e\notin K$ then $m_{e+1}$ is not
  cancelled at any stage greater than $t_0$ and is defined by stage
  $t$. If $e\in K$ then $m_{e+1}$ can possibly be cancelled after
  stage $t_0$ but no later than at stage $t_1$ and so once again its
  final value will be defined by stage $t$.
  
  Finally, by induction on $e$, we will show that there are $e$
  intervals $[k,2k+1]$ where the default color BLUE for all $x$ in the
  interval is never changed and $2k+1 \leq m_{e}$.  Assume inductively
  this is true for all $e' \leq e$ and let $s$ be the stage when
  $m_{e+1,s} = m_{e+1}$ is defined. By construction $m_{e+1,s} $ is
  defined as $2k_{e+1}+2$ for some fresh $k_{e+1}>m_e$.  The default
  color for all $x$ in the interval $[k_{e+1}, 2k_{e+1}+1]$ is never
  changed from BLUE. \hfill $\Box$


  \section{$2$-colorings} \label{sec:2}

  As we will see in this section $2$-colorings are very special.  For
  this section we will use BLUE and RED as our colors.  Perhaps one of
  the earliest published results on path decompositions is the
  following.

\begin{theorem}[\citet{MR0239997}]\label{finite}
  Every $2$-coloring $c$ of $[n]^2$ has a path decomposition.
\end{theorem}

\begin{proof}
  We prove this statement by induction on $n$.  Clearly the statement
  is true for $[2]^2$.  Assume $c$ is a $2$-coloring of $[n+1]^2$. In
  particular $c$ induces a $2$-coloring on the subgraph $[n]^2$.

  By induction there is a path decomposition of $[n]^2$.  So there is
  a RED path, $P_r$ and a BLUE path, $P_b$, such that, if $i< n$, then
  $i$ is on exactly one of $P_r$ or $P_b$.

  If $P_r$ is empty then $P_b$ and $\{n\}$ is a path decomposition for
  $c$.  Similarly if $P_b$ is empty.

  Let $x_r$ be the end of the RED path and let $x_b$ be the end of the
  BLUE path.  Look at the color of the edge between $x_r$ and $n$.  If
  it is RED then add $n$ to the end of $P_r$ to get a path
  decomposition for $c$.  Similarly, if the color of the edge between
  $x_b$ and $n$ is BLUE then add $n$ to the end of $P_b$.

  Otherwise look at the color of the edge between $x_r$ and $x_b$.  If
  this is RED add $x_b,n$ (in that order) to the end of the RED path
  and remove $x_b$ from the end of the BLUE path. We will say that
  $x_b$ \emph{switches to RED}. If $c(\{x_r,x_b\})$ is BLUE then add
  $x_r,n$ (in that order) to the end of the BLUE path and remove $x_r$
  from the end of the RED path. In this case $x_r$ switches
  \emph{switches to BLUE}. In all cases we have obtained a path
  decomposition of $[n+1]^2$, thereby completing the inductive step.
\end{proof}

We are going to improve this theorem to the following:

\begin{theorem} \label{bound} If $c:[\mathbb{N}]^2\rightarrow 2$ then
  there is a $\Delta^c_2$ Path Decomposition. In particular, if $c$ is
  computable then it has a path decomposition that is computable from
  the halting set $K$.
\end{theorem}

The rest of this section is devoted to the proof of this theorem. This
proof will be nonuniform. We will also discuss other issues along the
way. Our first goal is to understand why we cannot simply iterate
Theorem~\ref{finite} infinitely often to get such a proof. We need to
examine the path constructed in Theorem~\ref{finite} very closely.

\begin{definition}
  Suppose $\tilde{P}_b$ is a BLUE path and $\tilde{P}_r$ a RED path.
  Then the pair $(\tilde{P}_b, \tilde{P}_r)$ is a \emph{one-step path
    extension} of $(P_b,P_r)$ if exactly one of the following holds:
  \begin{enumerate}
  \item $\tilde{P}_b$ is $P_b$ with one additional element at the end
    and $\tilde{P}_r = P_r$, or
  \item $\tilde{P}_r$ is $P_r$ with one additional element at the end
    and $\tilde{P}_b = P_b$, or
  \item $\tilde{P}_b$ is the path $P_b$ with the last element $x_b$
    removed and $\tilde{P}_r$ is $P_r$ with $x_b$ and some integer $x$
    added in that order at the end, (in this case $x_b$
    \emph{switches} to RED), or
  \item $\tilde{P}_r$ is the path $P_r$ with the last element $x_r$
    removed and $\tilde{P}_b$ is $P_b$ with $x_r$ and some integer $x$
    added in that order at the end, (in this case $x_r$
    \emph{switches} to BLUE).
  \end{enumerate}
 
  $(\tilde{P}_b,\tilde{P}_r)$ is a \emph{path extension} of
  $(P_b, P_r)$ if it can be obtained from $(P_b,P_r)$ by a sequence of
  one-step path extensions.  Also, if $(\tilde{P}_b,\tilde{P}_r)$ is a
  {path extension} of $(P_b, P_r)$, $\tilde{P}_r={P}_r$, and $x$ is
  the last element of $\tilde{P}_b$, then we say that
  $(\tilde{P}_b,\tilde{P}_r)$ is a \emph{BLUE path extension of
    $(P_b, P_r)$ to $x$}.  We similarly define a \emph{RED path
    extension of $(P_b, P_r)$ to x}.
\end{definition}

Note that if $x$ is on one of $P_b$ or $P_r$ and
$(\tilde{P}_b,\tilde{P}_r)$ is a path extension, 
then $x$ is on one of $\tilde{P}_b$ or $\tilde{P}_r$.

The proof of Theorem~\ref{finite} shows:

\begin{lemma}\label{step}
  Given any two finite disjoint paths $P_b$ and $P_r$ and an integer
  $n$ not on either path, we can computably in $c$ find a one-step
  path extension $(\tilde{P}_b,\tilde{P}_r)$ such that $n$ appears on
  exactly one of these paths.
\end{lemma}

We cannot generalize this lemma for more than $2$ colors.
In fact, Theorem~\ref{finite} fails for more than $2$ colors.

\begin{theorem}[\citet{MR3194196}] \label{three} Given any $r > 2$
  there are infinitely many $m$ such that there is an  $r$-coloring $c$
  of $[m]^2$ without a path decomposition.
\end{theorem}

As our first attempt to prove Theorem~\ref{bound}, we will iterate
Lemma~\ref{step} infinitely often to build paths $P_{b,s}$ and
$P_{r,s}$ by stages.  Start with $P_{b,0} = P_{r,0} = \emptyset$.  At
stage $s+1$, apply Lemma~\ref{step} to $P_{b,s}$ and $P_{r,s}$ and the
least integer not on either path $n$ to get $P_{b,s+1}$ and
$P_{r,s+1}$.  Once we have constructed these sequences, we need a way
to extract from them two paths $P_b$ and $P_r$ and then try to argue
that they form a path decomposition. We can do this if the position of
every number eventually stabilizes. This idea is captured by the
following definition.

\begin{definition}
  Suppose that, for every natural number $s$, \\
  $P_s = x_0^s, \dots, x_{k_s}^s$ is a finite BLUE (RED) path.  We
  define the BLUE (RED) path
  $\lim_s P_s = x_0, x_1, \ldots, x_{n} ,\ldots $ by
  $x_n = \lim_s x_n^s$ as long as $\lim_s x_i^s$ exists for all
  $i \le n$.  If there is an $i \le n$ for which $\lim_s x_i^s$ does
  not exist, then we leave $x_n$ undefined.
\end{definition}

Given the sequences $\{ P_{b,s}\}_{s\in\mathbb{N}} $ and
$\{P_{r,s}\}_{s\in\mathbb{N}}$, we know that every $n$ eventually
appears on one of $P_{b,s}$ and $P_{r,s}$ at some stage $s$ and
remains on either $P_{b,t}$ and $P_{r,t}$ at all stages $t>s$.  So if
every $n$ only switches between the two sides finitely often then the
pair $P_b$ and $P_r$ is a path decomposition.  The limit will exist,
although we might not have an explicit way to compute the limit.

However, it is possible for an $n$ to switch infinitely often. It is,
in fact, even possible to build a $c$ such that every number $n$
switches sides infinitely often. For such a $c$, it would be the case
that $\limsup_s |P_{b,s}| =\infty$ but $\lim_s P_{b,s}$ is empty, and
likewise for $P_r$.


We have to alter our approach.  We will still build our path
decomposition as the stagewise limit of path extensions, although they
will no longer be one-step path extensions.  At each stage $s$ we will
have disjoint finite paths $P_{b,s}$ and $P_{r,s}$.  The pair
$(P_{b,s+1},P_{r,s+1})$ will be a path extension of
$(P_{b,s}, P_{r,s})$. The integers $x_{r,s}$ and $x_{b,s}$ will be the
ends of these paths at stage $s$.  When the stage is clear, we will
abuse notation and drop the $s$ in $x_{r,s}$ and $x_{b,s}$.  We will
need the following:

\begin{definition}
  Suppose $(\tilde{P}_b,\tilde{P}_r)$ is a one-step path extension of
  $(P_b,P_r)$ obtained by Case (3) (i.e.\ $x_b$ switches to RED and is
  followed by $x$ on $\tilde{P}_r$.)  We say that $x_b$ \emph{strongly
    switches} to RED if there is no BLUE path extension of $(P_b,P_r)$
  to $x$.  We similarly define what it means for $x_r$ to
  \emph{strongly switch} to BLUE.
  
  We say $(\tilde{P}_b,\tilde{P}_r)$ forms a \emph{strong one-step
    path extension} of $(P_b,P_r)$ if the pair
  $(\tilde{P}_b,\tilde{P}_r)$ is a one-step path extension of
  $(P_b, P_r)$ via either Cases~(1) or (2), or via Cases~(3) or (4)
  with a strong switch.
  
  We say that $(\tilde{P}_b,\tilde{P}_r)$ forms a \emph{strong path
    extension} of $(P_b,P_r)$ if it can be obtained from $(P_b,P_r)$
  by a sequence of strong one-step path extensions.
\end{definition}

The following lemma is the key combinatorial property that will
provide stability to constructions that are performed using strong
path extensions.

\begin{lemma}\label{strong}
  If $n$ strongly switches to RED then $n$ can never switch back to
  BLUE by a path extension. The RED path up to $n$ is stable.
\end{lemma}

Before we prove Lemma \ref{strong}, we require a more basic
order-preservation lemma concerning path extensions.

\begin{lemma}\label{order preserving}
  Assume $(\tilde{P}_b,\tilde{P}_r)$ is a path extension of
  $(P_b,P_r)$.
  Assume that $n$ and $m$ are two numbers such that one of the
  following holds.
  \begin{itemize}
  \item $n$ appears before $m$ in ${P}_r$.
  \item $m$ appears before $n$ in ${P}_b$.
  \item $n$ appears in ${P}_r$ and $m$ appears in ${P}_b$.
  \end{itemize}
  Then one of the following holds.
  \begin{itemize}
  \item $n$ appears before $m$ in $\tilde{P}_r$.
  \item $m$ appears before $n$ in $\tilde{P}_b$.
  \item $n$ appears in $\tilde{P}_r$ and $m$ appears in $\tilde{P}_b$.
  \end{itemize}
\end{lemma}

\begin{proof}
  The proof is an easy induction argument using the fact that only the
  last element of a path can ever switch to the other path. Any time
  the elements attempt to switch which path they are on, the latter
  element must switch before the earlier element.
\end{proof}

We now proceed to prove Lemma \ref{strong}.

\begin{proof}[Proof (Lemma \ref{strong}).]
  Assume not. Let $(P_{b,0},P_{r,0})$, $(P_{b,1},P_{r,1})$,
  $(P_{b,2},P_{r,2})$, $(P_{b,3},P_{r,3})$ be pairs of finite paths
  such that
  \begin{enumerate}
  \item $(P_{b,1},P_{r,1})$ is a one-step path extension of
    $(P_{b,0},P_{r,0})$ in which $n$ strongly switches to RED,
  \item $(P_{b,2},P_{r,2})$ is a path extension of $(P_{b,1},P_{r,1})$
    in which $n$ never switches,
  \item $(P_{b,3},P_{r,3})$ is a one-step path extension of
    $(P_{b,2},P_{r,2})$ in which $n$ switches to BLUE.
  \end{enumerate}
  Let $m$ be the element following $n$ in $P_{r,1}$.

  By definition of a strong switch, we have that there is no BLUE path
  extension of $(P_{b,0},P_{r,0})$ to $m$.  In particular, there is no
  BLUE path from $n$ to $m$ that does not involve any integers
  (besides $n$) from $P_{b,0}$ or $P_{r,0}$.

  By hypothesis, $n$ is in $P_{b,3}$, so by Lemma \ref{order
    preserving}, $m$ must appear before $n$ in $P_{b,3}$.

  But then $P_{b,3}$ has both $n$ and $m$ in it, and thus there is a
  BLUE path from $n$ to $m$. This provides a contradiction provided
  that we can prove that this path $P$ does not involve any integers
  besides $n$ from $P_{b,0}$ or $P_{r,0}$.

  To show this, note that $n$ and $m$ are the last two elements of
  $P_{r,1}$. In particular, every element of $P_{r,0}$ appears before
  $n$ in $P_{r,1}$, and so by Lemma \ref{order preserving}, if it is
  in $P_{b,3}$, it must appear after $n$. (This does not happen,
  although we do not need this fact for this proof.) Likewise, every
  element, besides $n$, of $P_{b,0}$ is in $P_{b,1}$. Therefore, by
  Lemma \ref{order preserving}, it must appear before $m$ in
  $P_{b,3}$.

  Note that we now have that if $n$ strongly switches to RED, then the
  RED path up to $n$ is stable: $n$ can never switch back to BLUE, and
  so nothing can be added to or removed from the RED path before $n$.
\end{proof}

Below we will modify the initial construction and require that all
switches be strong.  This avoids the problem of instability discussed
above. We will from now on use only strong path extensions.

We note here that if our goal was only to provide another proof of
Theorem \ref{original} for $r = 2$, then we would be done.  The
analogue of Lemma \ref{step} is not true with strong one-step
extensions, but it is true with strong extensions, so we could simply
use the initial construction with strong extensions to provide a path
decomposition for $c$. However, this process might not produce a
$\Delta^c_2$ path decomposition for the following reason.

If there are infinitely many strong RED switches, but only finitely
many BLUE switches, then the RED path is stabilized in a way that
allows it to be computed in a $\Delta^c_2$ manner, but the BLUE path
can only be computed in a $\Delta^c_3$ manner. It is $\Delta^c_2$ to
recognize a strong switch, but it is $\Delta^c_3$ to recognize that an
element will never strongly switch in the future. We will discuss this
in more detail in Section~\ref{oracle}.
For now, the key point is that we must sacrifice some of the
simplicity of the construction in order to provide a construction that
can be carried out by a computationally weaker oracle.


We now describe our construction explicitly. As suggested in the above
paragraph, the construction will depend on whether the number of
strong BLUE switches is finite or infinite and similarly for RED.
This leads us to a case-by-case analysis of our path decomposition.
In Section~\ref{uniform} we will show that there is no uniform way to
produce a $\Delta^c_2$ path decomposition, which implies that there is
no way to prove Theorem~\ref{bound} without some sort of case-by-case
analysis.

The following allows us to define our cases nicely.

\begin{definition}
  \label{state}
  For a coloring $c$, we will say that we can \emph{always strongly
    RED switch} if for every pair $(P_b, P_r)$ of disjoint finite BLUE
  and RED paths, there is a strong path extension
  $(\tilde{P}_{b}, \tilde{P}_{r})$ of $(P_b, P_r)$ such that there was
  a strong RED switch at some point during the path extension between
  $(P_b, P_r)$ and $(\tilde{P}_{b}, \tilde{P}_{r})$.

  We define being able to \emph{always strongly BLUE switch}
  similarly.
\end{definition}

\begin{lemma} \label{cannot} If the pair $({P}_b, {P}_r)$ witnesses
  that we cannot always strongly RED switch and
  $(\tilde{P}_b, \tilde{P}_r)$ is a strong path extension of
  $(P_b, P_r)$, then $(\tilde{P}_b, \tilde{P}_r)$ also witnesses that
  we cannot always strongly RED switch.
\end{lemma}

\begin{proof}
  Strong path extension is transitive. If there is a path extension of
  $(\tilde{P}_b, \tilde{P}_r)$ that includes a strong RED switch, then
  that same path extension is also a path extension of
  $({P}_b, {P}_r)$ that includes a strong RED switch.
\end{proof}
 
Our construction of a path decomposition breaks down into three
different procedures depending on whether or not we can always
strongly BLUE and RED switch.

\subsection{We can always strongly BLUE and RED switch}

We will inductively define $(P_{b,s}, P_{r,s})$ by multiple stages at
once. For each $s$, $(P_{b,s+1}, P_{r,s+1})$ will be a strong path
extension of $(P_{b,s}, P_{r,s})$.

Start with $P_{b,0}= P_{r,0} = \emptyset$.

Let $k$ the least stage where $(P_{b,k},P_{r,k})$ has yet to be
defined. Let $x$ be the least integer not on either of the paths
$P_{b,k-1}$ and $P_{r,k-1}$.  If there is a BLUE path extension to
$x$, let $(P_{b,k},P_{r,k})$ by that path extension. If this fails,
try the same for RED.  If both fail, switch either $x_b$ or $x_r$ as
in Lemma~\ref{step} to get $P_{b,k}$ and $P_{r,k}$. It follows that
this switch is a strong switch.  Next we stabilize some initial
segment of our paths: Let $(P_{b,k+1},P_{r,k+1})$ be a strong path
extension of $(P_{b,k},P_{r,k})$ that includes a RED switch, and let
$(P_{b,k+2},P_{r,k+2})$ be a strong path extension of
$(P_{b,k+1},P_{r,k+1})$ that includes a BLUE switch.


We then repeat for the next integer not yet on either path.

All switches are strong switches, and by Lemma~\ref{strong}, the
limits of these paths exist.  Since every integer is placed on our
paths at some stage, and since every integer can be switched at most
once, we have that every integer is on exactly one of the two limiting
paths.

Therefore this construction gives a path decomposition.

\subsection{We cannot always strongly RED switch}

Let $(P_{b,0}, P_{r,0})$ witness that we cannot always strongly RED
switch, and furthermore assume that among all such witnesses the
length of $P_{r,0}$ is minimal.

Now consider $(P_{b,s}, P_{r,s})$ and $x$ the least number not on
these two finite paths.  If there is a BLUE path extension, to $x$,
use that extension for $(P_{b,k+1},P_{r,k+1})$.  If this fails try to
do the same for RED.

We claim that one of the two options listed above will always
work. Towards a contradiction, suppose that both fail.  To add $x$ we
must switch (like in Theorem~\ref{finite}).  Call the resulting pair
$(\tilde{P}_b, \tilde{P}_r)$.  Then $(\tilde{P}_b, \tilde{P}_r)$ is a
strong extension of $(P_{b,s}, P_{r,s})$. So if $x_b$ switches to RED
then it strongly RED switches. This is not possible, by our choice of
$(P_{b,0}, P_{r,0})$.

It follows that $c(\{x_b, x_r\}) = BLUE$ and $x_r$ strongly switches
to BLUE. By Lemma~\ref{cannot}, since $(\tilde{P}_b, \tilde{P}_r)$ is
a strong extension of $(P_{b,s}, P_{r,s})$, it must also witness that
we cannot always strongly RED switch.  Note that $\tilde{P}_r$ is
shorter than $P_{r,s}$, so by the minimality of $P_{r, 0}$, we know
that $P_{r,s} \neq P_{r,0}$.  So $x_r$ was added to the RED path at
some stage $t \leq s$.

Hence there is no BLUE path from $x_{b,t}$ to $x_r$ which is otherwise
disjoint from $P_{b,t}$ and $P_{r,t}$, as otherwise $x_r$ would have
been added to $P_{b,t}$.

On the other hand there is a BLUE path from $x_{b,t}$ to $x_{b}$,
witnessed by the fact that $x_{b,t}$ and $x_{b}$ are both on the BLUE
path at stage $s$. (By induction, our construction has no switches up
to this point, so $x_{b,t}$ is still on the blue path at stage $s$.)
We also have that the pair $x_{b}, x_r$ is colored BLUE, so there is a
BLUE path from $x_{b,t}$ to $x_r$ disjoint from $P_{b,t}$ and
$P_{r,t}$, a contradiction.
\subsection{We cannot always strongly BLUE switch}

This case is dealt with in the same way as the previous one.


\subsection{The use of the oracle $c'$} \label{oracle}

The existence of a path from $x$ to $n$ is existential in the
coloring. The lack of a path from $x$ to $n$ is universal in the
coloring. So deciding if ``$\tilde{P}_b$ and $\tilde{P}_r$ is a
{one-step path extension} of $P_b$ and $P_r$ and $x_b$ strongly
switches to RED'' is universal in the coloring and so computable in
$c'$.

As a result, $c'$ can be used as an oracle to implement both of the
above constructions. In the case where we can always strongly RED and
BLUE switch, we can then use $c'$ to compute both of the paths because
both paths are stabilized by strong switches, and $c'$ can recognize
the strong switches. In the case where we cannot strongly RED (BLUE)
switch, we can also use $c'$ to compute both of the paths because both
paths are already stable: no numbers ever switch from one path to the
other.

Note that Definition~\ref{state} is $\Pi_3$ in the coloring. Our
division of cases depends on the truth of this statement and the
witness to its failure. This is finite information but as a result the
proof is not uniform in $c'$.  In Theorem \ref{nonuniform}, we will
show that this nonuniformity cannot be removed.

The more naive construction, always greedily adding the next element
by a strong extension with no case-by-case breakdown, can be
implemented uniformly by $c'$, but the construction could potentially
result in infinitely many RED switches and finitely many BLUE switches
(or vice-versa). In this case, the RED path would be computable from
$c'$ because the strong RED switches would stabilize it. On the other
hand, the BLUE path would not necessarily be computable from $c'$: the
statement that an initial segment of the BLUE path has stabilized is
universal in the construction (``for all future steps of the
construction, none of these elements ever RED switch'') and so is
$\Pi_2$ in the coloring. Thus, the naive construction could
potentially produce a path decomposition in which one of the two paths
is computable from $c''$, but not $c'$.

In our proof of Theorem~\ref{bound}, if we can always switch then both
the BLUE and RED paths are infinite.  But if we cannot always switch
one of the paths might be finite.  In that case the constructed paths
are both computable from the coloring. In Theorem~\ref{hardiseasy}, we
will see by a more delicate case breakdown that this actually always
happens. Thus, although our strongly switching proof does not work for
all colorings, it does work for all ``difficult'' colorings: colorings
for which there is no computable solution.

\section{Uniformity}\label{uniform}

In the above section, we have provided a nonuniform $\Delta^0_2$
construction and a uniform $\Delta^0_3$ construction of a path
decomposition. Furthermore, Theorem~\ref{bound} showed that, in
general, we cannot hope for a construction that is simpler than
$\Delta^0_2$, so the complexity of the construction cannot be
reduced. Here we address the question of whether the nonuniformity of
the $\Delta^0_2$ construction can be reduced.


Theorem~\ref{hardiseasy} shows that if our primary construction for
Theorem~\ref{bound} fails, then there must be a path decomposition for
$c$ in which one of the two paths is finite and the other is
computable from $c$.

Theorem~\ref{nonuniform} shows that there is no uniform $\Delta^0_2$
path decomposition.  Theorem~\ref{nonuniform2} improves this to show
we cannot even get by with a finite set of possible $\Delta^0_2$
indices for our path composition. Thus, nonuniformity is unavoidable.

In light of this, the result in Theorem~\ref{hardiseasy} is the
closest possible thing to a reduction of nonuniformity: All of the
noncomputable cases are handled by a single uniform $\Delta^0_2$
construction, which cannot have its complexity reduced due to
Theorem~\ref{bound}, and all of the nonuniform cases (unavoidable, by
Theorem~\ref{nonuniform2}) are handled by computable constructions
that are as simple as possible, with one path finite and the other
computable.

We should clarify precisely what we mean by a $\Delta^0_2$ path
decomposition and an index for such an object.

\begin{definition}
  A $\Delta^0_2$ path decomposition is a pair $(P_b,P_r)$ of partial
  $0'$-computable functions for which the domain of each is an initial
  segment of $\mathbb{N}$, the ranges partition $\mathbb{N}$, and such
  that for every $n+1 \in \dom{P_b}$, $c\{P(n), P(n+1)\}$ is BLUE, and
  similarly for $P_r$.

  A $\Delta^0_2$ index for a decomposition is a pair of numbers
  $(i_b, i_r)$ with $\Phi_{i_b}^{\emptyset'} = P_b$ and
  $\Phi_{i_r}^{\emptyset'} = P_r$.  Equivalently, by the limit lemma,
  it is a pair of numbers $(j_b, j_r)$ such that $\phi_{j_b}$ and
  $\phi_{j_r}$ are total computable functions,
  $P_b(x) = \lim_s \phi_{j_b}(x,s)$ and
  $P_r(x) = \lim_s \phi_{j_r}(x,s)$ for all $x$ (where the limit does
  not converge when $x$ is not in the domain).
\end{definition}


\subsection{When we cannot always strongly BLUE and RED switch}


\begin{theorem}\label{hardiseasy}
  There is a computable function $f$ with the property that for any
  $e$, if $e$ is an index for a computable coloring $c$, then either
  $f(e)$ is an index for a $\Delta^0_2$ path decomposition for $c$, or
  there is a computable path decomposition for $c$ in which one of the
  two paths is finite.
\end{theorem}

Note that this theorem relativizes: there is a uniform way to take a
coloring $c$, and attempt to produce a $\Delta^c_2$ path decomposition
so that either the attempt succeeds, or there is a $c$-computable path
decomposition for $c$ in which one of the two paths is finite.

\begin{proof}
  The proof hinges on the case analysis from the proof of
  Theorem~\ref{bound}.

  In the case where we can always strongly BLUE and RED switch, the
  proof is uniform: we alternate between adding the next element,
  adding a RED switch, and adding a BLUE switch, and our $\Delta^c_2$
  path decomposition is simply the path decomposition stabilized by
  the switches. Our function $f$ will be the function corresponding to
  attempting to do that construction.

  It remains to show that if we cannot always RED switch, then there
  is always a $c$-computable path decomposition in which one of the
  two paths is finite. (The case where we cannot always BLUE switch
  will follow by symmetry.)

  As in the proof of Theorem~\ref{bound}, let $(P_{b,0}, P_{r,0})$
  witness that we cannot always RED switch, and furthermore assume
  that among all such witnesses the length of $P_{r,0}$ is
  minimal. Let $x_b$ and $x_r$ be the endpoints of the paths
  $P_{b,0}, P_{r,0}$. We split our proof into two cases.

  \medskip

  \noindent {\bf Case 1:} Assume that for every $n\in\mathbb{N}$, and
  for every BLUE extension $(P_{b}, P_{r,0})$ of $(P_{b,0}, P_{r,0})$,
  if $n$ is not on either $P_{b}$ or $P_{r,0}$, then there is a BLUE
  path extension of $(P_{b}, P_{r,0})$ to $n$.

  In this case, we will use $P_{r,0}$ as our RED path, and grow our
  BLUE path to cover the rest of $\mathbb{N}$. We use a basic greedy
  algorithm.

  At stage $s$, let $(P_{b,s}, P_{r,0})$ be the pair of paths that we
  have, and let $n_s$ be the smallest number not on either path. We
  search for a BLUE path extension of $(P_{b,s}, P_{r,0})$ to $n_s$,
  and when we find such an extension, we let $(P_{b,s+1}, P_{r,0})$ be
  the first such extension that we find.

  By hypothesis, we will always find such an extension, and this
  algorithm clearly covers all of $\mathbb{N}$ in the limit.

  \medskip

  \noindent {\bf Case 2:} Assume Case 1 does not hold. Let
  $(P_{b,1}, P_{r,0})$ be a BLUE path extension of
  $(P_{b,0}, P_{r,0})$, and $n_0\in\mathbb{N}$ so that there is no
  BLUE path extension of $(P_{b,1}, P_{r,0})$ to $n_0$.

  We claim that in this case, we actually have that for every
  $n\in\mathbb{N}$, and for every RED extension $(P_{b_1}, P_{r})$ of
  $(P_{b,1}, P_{r,0})$, if $n$ is not on either $P_{b_1}$ or $P_{r}$,
  then there is a RED path extension of $(P_{b_1}, P_{r})$ to
  $n$. Thus we may use the RED analogue of the previous algorithm. The
  proof of this claim will be somewhat circuitous.

  First we show that there is no $n$ such that $(P_{b,1}, P_{r,0})$
  has both a RED extension to $n$ and a BLUE extension to $n$. After
  this we will show that actually either there is no $n$ such that
  $(P_{b,1}, P_{r,0})$ has a BLUE extension to $n$ or there is no $n$
  such that $(P_{b,1}, P_{r,0})$ has a RED extension to $n$. We will
  then show that the first case is true. Finally, from there, we will
  show that for every $n\in\mathbb{N}$, and for every RED extension
  $(P_{b_1}, P_{r})$ of $(P_{b,1}, P_{r,0})$ there is a RED path
  extension of $(P_{b_1}, P_{r})$ to $n$.

  During these proofs, we will repeatedly use the following facts.
  \begin{enumerate}
  \item $(P_{b,1}, P_{r,0})$ is a witness to the fact that we cannot
    always RED switch.
  \item $P_{r,0}$ has minimal length among the RED paths of such
    witnesses.
  \item There is no BLUE path extension of $(P_{b,1}, P_{r,0})$ to
    $n_0$.
  \end{enumerate}
  The first two facts follow from the fact that $(P_{b,1}, P_{r,0})$
  is a BLUE path extension of $(P_{b,0}, P_{r,0})$. The third is from
  our hypothesis in Case 2.

  \begin{claim}\label{claim1}
    There is no $n$ such that $(P_{b,1}, P_{r,0})$ has both a RED
    extension to $n$ and a BLUE extension to $n$.
  \end{claim}

\begin{proof}
  Assume not, and let $n_1$ be such an $n$. Replacing $n_1$ if
  necessary, we may assume that the BLUE path extension to $n_1$ and
  the RED path extension to $n_1$ do not intersect before
  $n_1$. Consider the edge from $n_1$ to $n_0$.

  If this edge is BLUE, then we can add $n_0$ to the end of the BLUE
  extension to $n_1$, creating a BLUE path to $n_0$. This contradicts
  fact (3).

  If the edge is RED, then consider the path extension
  $(\tilde{P}_b,\tilde{P}_r)$ of $(P_{b,1}, P_{r,0})$ in which
  $\tilde{P}_b$ is created by the BLUE path extension to $n_1$, and
  $\tilde{P}_r$ is created by the RED path extension to $n_1$, but
  with $n_1$ removed from the end. Let $x$ be the last element of
  $\tilde{P}_r$. Now we have that the edge from $x$ to $n_1$ is RED,
  the edge from $n_1$ to $n_0$ is RED, and there is no BLUE path
  extension from $(\tilde{P}_b,\tilde{P}_r)$ to $n_0$. So we can
  strongly switch $n_1$ to RED. This contradicts fact (1).
\end{proof}

A path extension $(\tilde{P}_b,\tilde{P}_r)$ of $({P}_b,{P}_r)$ is
\emph{nontrivial} if $(\tilde{P}_b,\tilde{P}_r)\neq({P}_b,{P}_r)$.

\begin{claim}\label{claim2}
  Either there is no $n$ such that $(P_{b,1}, P_{r,0})$ has a
  nontrivial BLUE extension to $n$ or there is no $n$ such that
  $(P_{b,1}, P_{r,0})$ has a nontrivial RED extension to $n$.
\end{claim}

\begin{proof}
  Assume not, and let $n_1,n_2$ be such that $(P_{b,1}, P_{r,0})$ has
  a nontrivial BLUE extension to $n_1$ and a nontrivial RED extension
  to $n_2$. Consider the edge between $n_1$ and $n_2$. If the edge is
  RED, then we can use it with the RED path to $n_2$ to create a RED
  path to $n_1$. But then $(P_{b,1}, P_{r,0})$ has both a RED
  extension to $n_1$ and a BLUE extension to $n_1$, contradicting
  Claim 1. If the edge is BLUE, we can similarly conclude that
  $(P_{b,1}, P_{r,0})$ has both a RED and a BLUE extension to $n_2$,
  again contradicting Claim~\ref{claim1}.
\end{proof}

\begin{claim}\label{claim3}
  There is no $n$ such that $(P_{b,1}, P_{r,0})$ has a nontrivial BLUE
  extension to $n$.
\end{claim}

\begin{proof}
  Again, assume not. Then by Claim~\ref{claim2}, there is no $n$ such
  that $(P_{b,1}, P_{r,0})$ has a nontrivial RED extension to $n$. In
  particular, $(P_{b,1}, P_{r,0})$ does not have a RED extension to
  $n_0$. By fact (3), $(P_{b,1}, P_{r,0})$ also does not have a BLUE
  extension to $n_0$. It follows, by Lemma~\ref{step}, that we can add
  $n_0$ by a switch, and the switch must be a strong switch.

  The switch cannot be a strong RED switch by fact (1). Also, the
  switch cannot be a strong BLUE switch because if we performed a
  strong BLUE switch on $(P_{b,1}, P_{r,0})$, it would create a strong
  path extension in which we decreased the length of the RED
  path. This contradicts fact (2), recalling that by Lemma
  \ref{cannot}, any strong path extension of $(P_{b,1}, P_{r,0})$ must
  also witness that we cannot always RED switch.
\end{proof}

\begin{claim}\label{claim4}
  For every $n\in\mathbb{N}$, and for every RED extension
  $(P_{b,1}, P_{r})$ of $(P_{b,1}, P_{r,0})$, if $n$ is not on either
  $P_{b,1}$ or $P_{r}$, then there is a RED path extension of
  $(P_{b,1}, P_{r})$ to $n$.
\end{claim}

\begin{proof}
  Let $(P_{b,1}, P_{r})$ be a RED extension of $(P_{b,1}, P_{r,0})$,
  and let $n_1$ be an element of $\mathbb{N}$ that is on neither
  $P_{b,1}$ nor $P_{r}$. Assume there is no RED extension of
  $(P_{b,1}, P_{r})$ to $n$. By Claim~\ref{claim3}, there is also no
  BLUE extension of $(P_{b,1}, P_{r})$ to $n$, because extending the
  RED path cannot make it any easier to find a BLUE extension. Then,
  by Lemma~\ref{step}, to add $n_1$ to $(P_{b,1}, P_{r})$, we must do
  a switch, and the switch must be strong.

  We split our proof into two cases:

  If $(P_{b,1}, P_{r})=(P_{b,1}, P_{r,0})$, then the proof proceeds
  exactly as the proof of Claim~\ref{claim3}. The switch cannot be a
  strong RED switch by fact (1), and the switch cannot be strong a
  BLUE switch by fact (2).

  If $(P_{b,1}, P_{r})\neq(P_{b,1}, P_{r,0})$, then the switch still
  cannot be a strong RED switch by fact (1), recalling again that by
  Lemma \ref{cannot}, any strong path extension of
  $(P_{b,1}, P_{r,0})$ must also witness that we cannot always RED
  switch. The switch also cannot be a BLUE switch, because for the
  switch to be a BLUE switch, the edge from the end of the $P_{b,1}$
  to the end of $P_{r}$ must be BLUE, contradicting Claim
  \ref{claim3}.
\end{proof}

This completes the proof of Theorem \ref{hardiseasy}.
\end{proof}

Note that if we are in Case 2 of the above construction, we have no
way of knowing whether we will find the $x$ we are looking for, so we
cannot know when to switch to the construction for Case 1. This does
not concern us. We are only proving that there is a computable path
decomposition with one path finite, not that the path decomposition
can be found uniformly.

\subsection{No Uniform $\Delta^0_2$ index}

\begin{theorem}\label{nonuniform}
  There is no partial computable function $f$ such that if $e$ is an
  index for a computable coloring $c$, then $f(e)$ is an index for a
  $\Delta^0_2$ path decomposition for $c$.
   \end{theorem}

This proof is somewhat more complicated than necessary, because it is
intended to serve as an introduction to the proof of
Theorem~\ref{nonuniform2}, and so we are performing a simplified
version of the construction found in that proof.

\begin{proof}\footnote{We want to thank
    an anonymous reader for pointing out an error in an early version of
    this proof.}
  Let $f$ be a partial computable function. We create a computable
  coloring with index $e$ such that if $f(e)$ is defined, then $f(e)$
  is not an index for a $\Delta^0_2$ path decomposition for $c$.

  Our construction is in stages. During stage $s$ of the construction,
  for every $t\leq s$, we color the pair $\{t,s+1\}$.

  By the recursion theorem, we may use an index, $e$, for the coloring
  that we are constructing.  We begin computing $f(e)$, and while we
  wait for it to halt, we color everything BLUE with everything else.

  If $f(e)$ halts, then its value provides us a pair of Turing
  reductions $\Phi_{i_b}$ and $\Phi_{i_r}$, and by the limit lemma we
  may obtain a pair of total computable functions $\phi_b$ and
  $\phi_r$ with $\lim_s \phi_b(x,s) = \Phi_{i_b}^{\emptyset'}(x)$ for
  all $x$, and similarly for $\phi_r$.  We define
  $P_b = \Phi_{i_b}^{\emptyset'}$, $P_{b,s}(x) = \phi_b(x,s)$, and
  similarly for $P_r$ and $P_{r,s}$.  We will assume that
  $P_{b,s}(x) \le s$ and $P_{r,s}(x) \le s$ for every $x$ and $s$.

  We will have two strategies, $S_b$ and $S_r$, which work to ensure
  that if $(P_b,P_r)$ forms a path decomposition for $c$, then $P_b$
  (respectively $P_r$) is finite.  Note that if both $S_b$ and $S_r$
  achieve their goals, then $(P_b, P_r)$ cannot form a path
  decomposition.
  
  One of $S_b$ and $S_r$ will have high priority, and the other will
  have low priority, but this priority assignment will potentially
  change over the course of the construction.  High priority goes to
  the strategy whose corresponding first element has been stable the longest.
  To make this precise, at stage $s$, define $t(b,s)$ to be least such
  that for every $t \in [t(b,s), s]$, $P_{b,t}(0) = P_{b,s}(0)$, and
  make the similar definition for $t(r,s)$.  Then $S_b$ has high
  priority at stage $s$ if $t(b,s) \le t(r,s)$, and otherwise $S_r$
  has high priority.
  
  Suppose first that $S_b$ has high priority at stage $s$.  Let $s_0$
  be least such that $s_0 \ge t(b,s)$ and $S_b$ has high priority at
  stage $s_0$.  Note that $S_b$ necessarily had high priority at every
  stage between $s_0$ and $s$.  For every $t \le s_0$, we color the
  pair $\{t, s+1\}$ RED.  This completes the action for $S_b$.
  
  We now describe the behavior of $S_r$ when it is of lower priority
  at stage $s$.  We consider whether there are $k,\ell \le s$ with
  \[
    \text{range}(P_{b,s}\!\!\upharpoonright_\ell) \sqcup
    \text{range}(P_{r,s}\!\!\upharpoonright_k) \supseteq [0,s_0]
  \] and $P_{r,s}(k) > s_0$.  If there are no such $k$ and $\ell$, we
  take no action for $S_r$ at stage $s$.  If there are such a $k$ and
  $\ell$, we fix the least such.  Let $s_1$ be least such that

  $s_1 \ge P_{r,s}(k)$,
  $P_{b,t}\!\!\upharpoonright_\ell = P_{b,s}\!\!\upharpoonright_\ell$
  and
  $P_{r,t}\!\!\upharpoonright_{k+1} =
  P_{r,s}\!\!\upharpoonright_{k+1}$ for every $t \in [s_1, s]$.  For
  every $t \in (s_0, s_1]$, we color the pair $\{t,s+1\}$ BLUE.  This
  completes the action for $S_r$.  This a place where our construction
  is more complicated than necessary as it would be enough to color
  BLUE all relevant pairs not yet colored RED.
 
  If instead $S_r$ has high priority at stage $s$, we proceed as
  above, mutatis mutandis.
 
  At the end of stage $s$, for any $t\le s$ for which we have not yet
  colored $\{t,s+1\}$, we assign a color arbitrarily (for
  definiteness, BLUE).  This completes the construction.  We now
  verify that if $f(e)$ converges, it does not specify a path
  decomposition for $c$.
 
  Assume, towards a contradiction, that $(P_b, P_r)$ form a path
  decomposition for $c$.  Then at least one of $P_b$ and $P_r$ is
  nonempty, and for this path, its value at 0 (the first element of
  the path) will eventually converge.  So eventually one of the
  strategies will have high priority for cofinitely many stages with
  an unchanging least $s_0$.  Without loss of generality, assume this
  is $S_b$.  Then $P_b(0) \le s_0$, and by construction $c\{x,y\}$ is
  RED for every $x \le s_0 < y$.  So $P_b$ cannot contain any elements
  larger than $s_0$, and so $S_b$ has ensured its requirement by
  making $P_b$ finite.
 
  Since $(P_b, P_r)$ form a path decomposition, in particular their
  ranges cover $\mathbb{N}$.  So there are some least $k$ and $\ell$ with
  \[
    \text{range}(P_{b}\!\!\upharpoonright_\ell) \sqcup
    \text{range}(P_{r}\!\!\upharpoonright_k) \supseteq [0,s_0]
  \]
  and $P_r(k) > s_0$.  Let $s_1 \le P_r(k)$ be least such that
  $P_b\!\!\upharpoonright_\ell$ and $P_r\!\!\upharpoonright_{k+1}$ have
  converged by stage $s_1$.  Then $S_r$ will eventually select this
  $k,\ell$ and $s_1$ for cofinitely many stages.  By construction,
  $c\{x,y\}$ is BLUE for every $s_0 < x \le s_1 < y$, and
  $P_r(k) \in (s_0, s_1]$.  $P_r$ cannot contain any elements of
  $[0, s_0]$ after $P_r(k)$, since those elements have all either
  occurred on $P_b$ or on $P_r$ before $P_r(k)$.  Thus $P_r$ after
  $P_r(k)$ is entirely contained in $(s_0, s_1]$, and so $S_r$ has
  ensured its requirement of making $P_r$ finite.\end{proof}

We remark now that the technique of Theorem~\ref{nonuniform} only
allows us to build a coloring $c$ that can defeat any single uniform
manner of attempting to produce a $\Delta^0_2$ path decomposition from
$c$. For the coloring $c$ that we create, it is not at all difficult
to produce a path decomposition; it is just the case that the $f(e)$th
$\Delta^0_2$ path decomposition fails to do so.


\subsection{No finite set of $\Delta^0_2$ indices}

We show now, by a strengthening of the argument from
Theorem~\ref{nonuniform} that it is impossible to reduce the
nonuniformity to a finite collection of $\Delta^0_2$ indices.

\begin{theorem}\label{nonuniform2}
  There is no partial computable function $f$ such that if $e$ is an
  index for a computable coloring $c$, then $f(e)$ is an index for a
  finite c.e. set $W_{f(e)}$ one of whose elements is an index for a
  $\Delta^0_2$ path decomposition for $c$.
\end{theorem}

\begin{proof}
  Let $f$ be a partial computable function. We create a computable
  coloring with index $e$ such that if $f(e)$ is defined, and if
  $W_{f(e)}$ is finite, then no element of $W_{f(e)}$ is an index for
  a $\Delta^0_2$ path decomposition for $c$.

  Our construction is in stages. During stage $s$ of the construction,
  for every $t\leq s$, we color the pair $\{t,s+1\}$.

  As in the proof of Theorem~\ref{nonuniform}, we use an index, $e$,
  for the coloring that we are constructing.  We begin computing
  $f(e)$, and while we wait for it to halt, we color everything BLUE
  with everything else.

  If $f(e)$ halts, then we begin enumerating $W_{f(e)}$. Let
  $W_{f(e),s}$ be the stage $s$ approximation to $W_{f(e)}$.

  For each $i$ in $W_{f(e)}$, let $P_{b,i}$ and $P_{r,i}$ be the
  potential paths given by index $i$, and let $P_{b,i,s}$ and
  $P_{r,i,s}$ be their stage $s$ approximations, as in the proof of
  Theorem~\ref{nonuniform}. We again assume
  $P_{b,i,s}(x), P_{r,i,s}(x) \le s$ for all $x$.

  We will use $z$ as a variable for a color (BLUE or RED), as well as
  for a letter for a color ($b$ or $r$). We will write $1-z$ to refer
  to the other color.

  For each $i\in W_{f(e)}$, we will have two strategies $S_{b,i}$ and
  $S_{r,i}$, which will work to ensure that if $(P_{b,i},P_{r,i})$
  forms a path decomposition for $c$, then $P_{b,i}$ (respectively
  $P_{r,i}$) is finite. Note that if both $S_{b,i}$ and $S_{r,i}$
  achieve their goals, then $(P_{b,i},P_{r,i})$ cannot form a path
  decomposition for $c$.



  We will again arrange our strategies in a priority ordering based on
  length of stability of their first element.  That is, at stage
  $s$, define $t_0(z,i,s)$ to be least such that for every
  $t \in [t_0(z,i,s), s]$, $P_{z,i,t}(0) = P_{z,i,s}(0)$.  From
  $z \in \{b,r\}$ and $i \in W_{f(e), s}$, choose a pair $(z_0,i_0)$
  with $t_0(z,i,s)$ least (deciding ties by G\"odel numbering) to be
  our strategy of highest priority at stage $s$, and let
  $t_0(s) = t_0(z_0,i_0,s)$.

  Let $s_0$ be least such that $s_0 \ge t_0(s)$ and $S_{z_0,i_0}$ has
  highest priority at stage $s_0$.  For every $t \le s_0$, we color
  the pair $\{t,s+1\}$ with color $1-z_0$.  This completes the action
  for $S_{z_0,i_0}$.

  Next, for every pair $(z,i)$ other than $(z_0, i_0)$, consider
  whether there are $k,\ell \le s$ with
  \[
    \text{range}(P_{z,i,s}\!\!\upharpoonright_k) \sqcup
    \text{range}(P_{1-z,i,s}\!\!\upharpoonright_\ell) \supseteq
    [0,s_0]
  \]
  and $P_{z,i,s}(k) > s_0$.  For those pairs for which there are such
  $k$ and $\ell$, fix the least such $k$ and $\ell$ and let $t_1(z,i,s)$ be least
  such that
  $t_1(z,i,s) \ge P_{z,i,s}(k), P_{1-z,i,t}\!\!\upharpoonright_\ell =
  P_{1-z,i,s}\!\!\upharpoonright_\ell$ and
  $P_{z,i,t}\!\!\upharpoonright_{k+1} =
  P_{z,i,s}\!\!\upharpoonright_{k+1}$ for every
  $t \in [t_1(z,i,s), s]$.

  From those pairs with $t_1(z,i,s)$ defined, choose a pair
  $(z_1,i_1)$ with $t_1(z,i,s)$ least (deciding ties by G\"odel
  numbering) to be the strategy of next highest priority at stage $s$,
  and let $t_1(s) = t_1(z_1,i_1,s)$.  Let $s_1$ be least such that
  $s_1 \ge t_1(s)$ and $S_{z_1, i_1}$ has second highest priority at
  stage $s_1$.  For every $t \in (s_0, s_1]$, we color the pair
  $\{t,s+1\}$ with color $1-z_1$.  This completes the action for
  $S_{z_1,i_1}$.

  We continue in this fashion until we reach a $j$ where $t_j(z,i,s)$
  is not defined for any pair $(z,i)$.  We then color $\{t,s+1\}$ BLUE
  for any remaining $t \le s$ and end the stage.  This completes the
  construction.

  We now verify that if $W_{f(e)}$ is finite, then for every
  $i \in W_{f(e)}$, $(P_{b,i}, P_{r,i})$ is not a path decomposition
  for $c$.  Note that for every pair $(z,i)$, $t_j(z,i,s)$ is
  non-decreasing in $s$, and if $t_j(z,i,s)$ is undefined, then for
  all $t > s$ with $t_j(z,i,t)$ defined, $t_j(z,i,t) > s$.  It follows
  that the same holds for $t_j(s)$.

  Define $m$ to be greatest such that for all $j < m$,
  $t_j = \lim_s t_j(s)$ converges.  Let $(z_j, i_j)$ be the pair
  chosen for priority $j$ for cofinitely many stages (the pair
  defining $t_j = t_j(z_j,i_j,s)$ for almost every $s$).  The
  existence of such a pair follows from the above discussion, along
  with the assumption that $W_{f(e)}$ is finite. Let $k_j$ be the
  value $k$ chosen for this pair at cofinitely many stages.

  The following two claims will complete the proof of the result.

\begin{claim}
  For $j < m$, if $P_{z_j, i_j}$ is a monochromatic path with color
  $z_j$ and disjoint from $P_{1-z_j, i_j}$, then it is finite.
\end{claim}

\begin{proof}
  By a simple induction, the values $s_{j-1}$ and $s_j$ are eventually
  chosen the same at cofinitely many stages $s$ (taking
  $s_{-1} = -1$).  By construction, $c\{x,y\}$ is $1-z_j$ for every
  $s_{j-1} < x \le s_j < y$.  By our choice of $(z_j, i_j)$, every
  point in $[0, s_{j-1}]$ lies either on $P_{1-z_j, i_j}$ or is one of
  $P_{z_j, i_j}(0), \dots P_{z_j, i_j}(k_j-1)$.  Also,
  $P_{z_j, i_j}(k_j) \in (s_{j-1}, s_j]$.  So after
  $P_{z_j, i_j}(k_j)$, $P_{z_j, i_j}$ cannot contain any elements
  outside of $(s_{j-1}, s_j]$, and so must be finite.
\end{proof}

\begin{claim}
  If $i \in W_{f(e)}$ and $(z,i)$ is not one of the $(z_j, i_j)$ for
  any $j < m$, then $P_{z, i}$ is finite or
  $\text{range}(P_{z,i})\sqcup \text{range}(P_{1-z,i})$ is not all of
  $\mathbb{N}$.
\end{claim}

\begin{proof}
  Suppose not.  Since
  $\text{range}(P_{z,i})\sqcup \text{range}(P_{1-z,i})$ is all of
  $\mathbb{N}$ and $P_{z,i}$ is infinite, there are $k$ and $\ell$
  with
  \[
    \text{range}(P_{z,i}\!\!\upharpoonright_k) \sqcup
    \text{range}(P_{1-z,i}\!\!\upharpoonright_\ell) \supseteq [0,s_m]
  \]
  and $P_{z,i}(k) > s_m$, where $s_m$ is the value chosen for
  $(z_m, i_m)$ at cofinitely many stages.  At sufficiently large
  stages, $P_{z,i}\!\!\upharpoonright_{k+1}$ and
  $P_{1-z,i}\!\!\upharpoonright_k$ will converge, and $i$ will have
  appeared in $W_{f(e)}$, and so $t_{m+1}(z,i,s)$ will be defined and
  unchanging at sufficiently large $s$.  So $t_{m+1}(s)$ will be
  defined and bounded by $t_{m+1}(z,i,s)$ at all of these stages.
  Since $t_{m+1}(s)$ is nondecreasing, it must have a limit, contrary
  to our choice of $m$.
\end{proof}

Theorem \ref{nonuniform2} now follows: for any $i \in W_{f(e)}$,
either $P_{b,i}$ and $P_{r,i}$ are both finite, one of $P_{b,i}$ or
$P_{r,i}$ fails to be a monochromatic path of the appropriate color,
or there are elements of $\mathbb{N}$ which appear on neither or both paths.
In all cases, $(P_{b,i}, P_{r,i})$ is not a path decomposition for
$c$.
\end{proof}

\section{$\mathsf{RPD}$ compared to Ramsey's Theorem for pairs}\label{sec:ramsey}

One fact about (infinite) Ramsey's Theorem that is regularly used is
that for every coloring $c:[\mathbb{N}]^2 \rightarrow r$ and every
infinite set $X$, there is an infinite homogeneous set
$H \subseteq X$. However, a path decomposition of a set
$X\subseteq \mathbb{N}$ for the restricted coloring
$c:[X]^2 \rightarrow r$ does not help us to find a path decomposition
for the unrestricted coloring $c:[\mathbb{N}]^2 \rightarrow r$.

There is a proof which uses compactness to show the infinite version
of Ramsey's theorem implies the finite version.  For example, see
\citet{MR1044995}.  By Theorem~\ref{three}, we know this compactness
argument fails for the Rado Path Decomposition Theorem.  A compactness
argument breaks down since the paths linking numbers below $m$ might
also involve numbers larger than $m$.

\section{Corollaries in Mathematical Logic}\label{sec:logic}

For a reference for the terms used in this section we suggest
\citet{MR3244278}.  The existence of a non-principal ultrafilter on
the natural numbers is a strong assumption that unfortunately cannot
be shown in Zermelo Fraenkel set theory, see \citet{MR0176925}; the
axiom of choice is sufficient see \citet{MR1940513}.
By independent results of \citet{MR3225587}, \citet{MR2262315}, and
\citet{MR2950192}, the ultrafilter proof of Rado Path Decomposition
implies that for every $r$-coloring $c$ of $[\mathbb{N}]^2$ there is a
path decomposition arithmetical in $c$, and as a statement of second
order arithmetic the Rado Path Decomposition Theorem holds in  $\mathsf{ACA}_0$.

The same result can be obtained by an examination of the cohesive
proof in Section~\ref{Cohesive_Proof_Subsection}.  In fact, that proof
can give us more.  A careful analysis shows that a path decomposition
can always be found in the jump of the cohesive set $C'$. The key
issue is that exactly one $N(m,j)$ is large (with respect to our
cohesive set $C$).  It is $\Delta^C_2$ (but not computable in $C$) to
determine which one.

\citet{Jockusch.Stephan:93} have showed that $\mathbf{d}$ is PA over
${0'}$ if and only if there is a $C$ which is cohesive with respect to the
collection of all computable sets and $C' \leq_T \mathbf{d}$.  As
there is a $\mathbf{d}$ which is PA over $0'$ with
$\mathbf{d}' = 0''$, it follows that that there is always a path
decomposition whose jump is bounded by $c''$.

For $2$-colorings, Theorem~\ref{bound} shows this bound can be
improved to $\Delta^c_2$.  For stable colorings the bound can also be
improved to $\Delta^c_2$.  (Use the stable proof of $\mathsf{RPD}$ and note that
determining $m$'s color is $\Delta^c_2$.)
  
Theorem~\ref{halting} shows that we cannot expect to do better than
$\Delta^c_2$. So for stable and $2$-coloring the bound of $\Delta^c_2$
is tight.

For more than two colors, we do not have an exact calibration of the
effectivity of path decomposition.


\begin{question}
  Does every $3$-coloring $c$ have a $\Delta^c_2$ path decomposition?

\end{question}

  \begin{question}\label{PA}
    Is there an unstable $3$-coloring $c$ such that every path
    decomposition is PA over $0'$?
  \end{question}

\begin{question}
  Does increasing the number of colors past $3$ have any effect on the
  above two questions?

\end{question}

Theorem~\ref{halting} shows that as a statement of second order
arithmetic the Rado Path Decomposition Theorem implies
$\mathsf{ACA}_0$ over $\mathsf{RCA}_0$.  One can observe that the only
induction used is $\Sigma^0_1$ and hence available in $\mathsf{RCA}_0$.

One might wonder why we cannot use the generic construction to answer
Question~\ref{PA} by building a path decomposition that avoids the
cone of degrees above $\mathbf{0'}$.  The problem is that forcing
$\Sigma^G_1$ statements (like does $\Phi^G(w) \converge$) is
$\Sigma^X_2$ not $\Sigma^X_1$.  The ends of finite paths $P_j$ must
have color $j$ and determining this is not $\Sigma^X_1$.

\bibliographystyle{plainnat} \bibliography{rado}

\end{document}